\documentclass[reqno,a4paper,11pt]{amsart}
\usepackage{blindtext}

\usepackage{hyperref}

\usepackage{enumerate}
\usepackage{verbatim}

\usepackage[all]{xy}
\usepackage[usenames,dvipsnames]{xcolor}
\usepackage{tikz}
\usepackage{tikz-cd}

\usepackage{pgfplots}
\usetikzlibrary{3d}
\usetikzlibrary{shapes,decorations.pathreplacing}
\usetikzlibrary{calc}
\usetikzlibrary{backgrounds}
\usetikzlibrary{automata}
\usetikzlibrary{positioning}
\usetikzlibrary{decorations.markings}
\usetikzlibrary{arrows}
\usetikzlibrary{scopes}
\usetikzlibrary{intersections}
\tikzstyle myBG=[line width=3pt,opacity=1]

\newcommand{\ZM}{{\mathbb{Z}M}}

\newcommand{\Z}{\mathbb{Z}}

\newcommand{\FPn}{{\rm FP}\sb n}

\newcommand{\FP}{{\rm FP}}

\definecolor{cof}{RGB}{219,144,71}
\definecolor{pur}{RGB}{186,146,162}
\definecolor{greeo}{RGB}{91,173,69}
\definecolor{greet}{RGB}{52,111,72}

\usepackage{amssymb}
\usepackage{amsgen}
\usepackage{amsmath}
\usepackage{amsthm}
\usepackage{mathrsfs}
\usepackage{cite}
\usepackage{amsfonts}

\newcommand{\ov}[1]{\ensuremath{\overline {#1}}}

\newtheorem{Thm}{Theorem}
\newtheorem{Prop}[Thm]{Proposition}

\newtheorem{Lemma}[Thm]{Lemma}
{\theoremstyle{definition}
}
{\theoremstyle{remark}
}
{\theoremstyle{remark}
}
\newtheorem{Cor}[Thm]{Corollary}

{\theoremstyle{remark}
}

{\theoremstyle{remark}
}
{\theoremstyle{remark}
}
{\theoremstyle{remark}
}

\numberwithin{equation}{section}

\usepackage{graphicx}

\title[Free inverse monoids]{
Free inverse monoids are 
not \boldmath{$\mathrm{FP}_{2}$}
}

\keywords{
Free inverse monoid,
finitely presented, 
homological finiteness property.  
\\
\indent
This work was supported by the
EPSRC grant EP/N033353/1 `Special inverse monoids: subgroups, structure, geometry, rewriting systems and the word problem'. The second author was supported by  NSA MSP \#H98230-16-1-0047 and a PSC-CUNY award
This work was supported by the London Mathematical Society Research in Pairs (Scheme 4) grant 
(Ref:\ 41844), which funded a 6-day research visit of the second named author to the University of East Anglia (September 2019).
}

\begin{document}
\maketitle

\begin{center}
ROBERT D. GRAY
\footnote{School of Mathematics, University of East Anglia, Norwich NR4 7TJ, England.
Email \texttt{Robert.D.Gray@uea.ac.uk}.
}
and
BENJAMIN STEINBERG\footnote{
Department of Mathematics, City College of New York, Convent Avenue at 138th Street, New York, New York 10031,  USA.
Email \texttt{bsteinberg@ccny.cuny.edu}.}
\\
\end{center}

\begin{abstract}
We give a topological proof that a free inverse monoid on one or more generators is neither of type left-$\FP_2$ nor right-$\FP_2$. This strengthens a classical result of Schein that such monoids are not finitely presented as monoids.
\end{abstract}

\ 

Given how easy it is to prove that a group $G$ is finitely presented as a group if and only if it is finitely presented as a monoid, it is rather surprising  
that the same is not true of inverse monoids. 
Indeed it 
is a classical 
result of Schein~\cite{Schein1975} that free inverse monoids on a non-empty set of generators are not finitely presented as monoids.

Our goal in this paper is to prove the following stronger result about free inverse monoids.

\begin{Thm}\label{t:notfp2}
A free inverse monoid on one or more generators is neither of type left-$\FP_2$ nor right-$\FP_2$.
\end{Thm}

The free inverse monoid is an object of central importance in inverse semigroup theory. 
Free inverse monoids were studied in detail in classical work of 
Munn \cite{Munn74} and Scheiblich \cite{Scheiblich73}.
It follows from this work that 
the word problem is decidable for this class.
For a general introduction to the theory of inverse monoids we refer the reader to~\cite{Lawson}.

Recall that a monoid $M$ is said to be of type left-$\FPn$ if there is a projective resolution $P = (P_i)_{i \geq 0}$ of the trivial left $\ZM$-module $\Z$ such that $P_i$ is finitely generated for $i \leq n$.
There is a dual notion of right-$\FPn$, and we say a monoid is of type $\FPn$ if it is both of type left- and right-$\FPn$.
It is well known (see e.g. \cite{Otto1997}) that every finitely presented monoid is of type 
left- and right-$\FP_2$. Hence an immediate corollary of Theorem~\ref{t:notfp2} is Schein's theorem  \cite{Schein1975} that free inverse monoids on a non-empty set of generators are not finitely presented.

\begin{Cor}\label{Schein1975}
Free inverse monoids on one or more generators are not finitely presented.
\end{Cor}

Since inverse monoids are isomorphic to their duals, it suffices to show that $M$ is not of type left-$\FP_2$, which henceforth shall be called simply $\FP_2$. Pride~\cite{Pride2006} showed that the class of monoids of type $\FP_2$ is closed under taking retracts.  Since the free monogenic inverse monoid $M$  is a retract of any free inverse monoid on a non-empty set of generators, it suffices to prove that $M$ is not of type $\FP_2$. Our aim is now to prove the following theorem.

\begin{Thm}\label{t:monogenic.case}
The free monogenic inverse monoid is not of type $\FP_2$.
\end{Thm}

Before proving this result we briefly review some facts about free inverse monoids and the representation of their elements via Munn trees. For a full account of this theory we refer the reader to~\cite[Chapter~6]{Lawson}. Let $X$ be a non-empty set and let $X^{-1}$ be a set disjoint from $X$ and in bijective correspondence with $X$ via $x \mapsto x^{-1}$. The free inverse monoid $\mathrm{FIM}(X)$ is defined to be $Y / \rho$ where $Y = (X \cup X^{-1})^*$ and  $\rho$ is the congruence generated by the set
\[
\{ (ww^{-1}w,w): w \in Y \}
\cup
\{ (ww^{-1} zz^{-1}, zz^{-1} ww^{-1}): w,z \in Y^* \}.
\]
For each word $u \in Y$ we associate a tree $\mathrm{MT}(u)$, called the \emph{Munn tree}, of $u$ where $u$ is obtained by tracing the word $u$ in the Cayley graph $\Gamma(\mathrm{FG}(X))$ of the free group $\mathrm{FG}(X)$ with respect to the generating set $X$. 
Recall that if $M$ is a monoid and $A\subseteq M$, then the (right) \emph{Cayley digraph} $\Gamma(M,A)$ of $M$ with respect to $A$ is the graph with vertex set $M$ and with edges in bijection with $M\times A$ where the directed edge (arc) corresponding to $(m,a)$ starts at $m$ and ends at $ma$.  
So $\mathrm{MT}(u)$ is a finite birooted subtree of  $\Gamma(\mathrm{FG}(X))$ with initial vertex $1$ and terminal vertex the reduced form $r(u)$ of the word $u$ in the free group. Munn's solution to the word problem in $\mathrm{FIM}(X)$ says that $u=v$ in $\mathrm{FIM}(X)$ if and only if $\mathrm{MT}(u) = \mathrm{MT}(v)$ as birooted trees.

Now we turn our attention to the special case of the free monogenic inverse monoid and the proof of Theorem~\ref{t:monogenic.case}. 
For the remainder of this article, let $M$ denote the free monogenic inverse monoid. 
Let $x$ be the free generator of $M$ and let $y$ denote its
(generalized) inverse. 
Let $\Gamma$ be the Cayley digraph of $M$ with
respect to the generating set $\{x,y\}$.  
Then $M$ acts on the left of
$\Gamma$ by cellular mappings.  The augmented cellular chain complex
of $\Gamma$ gives a partial resolution of the trivial module
\[C_1(\Gamma)\xrightarrow{\,\,d_1\,\,} C_0(\Gamma)\xrightarrow{\,\,\varepsilon\,\,} \mathbb Z\longrightarrow 0.\]
Moreover, since the vertices of $\Gamma$ form a free $M$-set on $1$ generator (the vertex $1$) and the edges form a free $M$-set on $2$ generators (the arrows $1\xrightarrow{\,\, x\,\,}x$ and $1\xrightarrow{\,\,y\,\,} y$), this is, in fact, a partial free resolution which is finitely generated in each degree.  Therefore, if $M$ is of type $\FP_2$, we must have that $\ker d_1=H_1(\Gamma)$ is finitely generated as a $\mathbb ZM$-module (by~\cite[Proposition~4.3]{BrownCohomologyBook}). So our goal now is to show that $H_1(\Gamma)$ is not finitely generated as a $\mathbb ZM$-module. We remark that $H_1(\Gamma)$ is isomorphic as a $\mathbb ZM$-module to the relation module of $M$ in the sense of Ivanov~\cite{Ivanov}; see~\cite[Section~6]{GraySteinberg3}.

If $p$ is a path in $\Gamma$, there is a corresponding element $\ov p$ of $C_1(\Gamma)$ which is the weighted sum of the edges traversed by $p$, where an edge receives a weight of $n-k$ if it is traversed $n$ times in the forward direction and $k$ times in the reverse direction.

If $T$ is a spanning tree for $\Gamma$ (and we will choose a particular one shortly), then $H_1(\Gamma)$ is a free abelian group with a basis in bijection with the directed edges of $\Gamma\setminus T$.  If $v,w$ are vertices, then $[v,w]$ will denote the geodesic in $T$ from $v$ to $w$.  The basis element $b_e$ of $H_1(\Gamma)$ corresponding to a directed edge $e$ of $\Gamma\setminus T$ is $\ov{[1,\iota(e)]e[1,\tau(e)]^{-1}}$ where $\iota,\tau$ denote the initial and terminal vertex functions, respectively.  If $p$ is a closed path in $\Gamma$, then the homology class of $\ov{p}$ is the weighted sum of the basis elements $b_e$ where the weight of $b_e$ is $n-k$ with $n$  the number of traversals of $e$ by $p$ in the forward direction and $k$ the number of traversals in the reverse direction.

Our spanning tree $T$ will come from a prefix-closed set of normal forms for $M$ based on a right-left-right sweep of the Munn tree of an element.
Note here that the Munn trees are all subtrees of the Cayley graph of the infinite cyclic group with respect to the generators $x$ and $y=x^{-1}$.
The idea is we first sweep to the right in the Munn-tree as far as possible, then to the left as far as possible, and then, if necessary, back to the right.

 We will consider Munn trees for $M$ of two types.  A Munn tree is of \emph{Type I} if all its edges appear to the right of the in-vertex (where we view the Munn tree as embedded in the Cayley graph of $\mathbb Z$ with the in-vertex at $0$); the Munn tree of the empty word is vacuously considered of Type I.  The normal form for such an element is $x^ny^k$ with $0\leq k\leq n$ where $n$ is the number of edges in the Munn tree and the out-vertex is $k$ to the left of the rightmost vertex.

 A Munn tree is of \emph{Type II} if it contains an edge to the left of the in-vertex. The normal form of such a Munn tree is of the form $x^ny^kx^j$ where $0\leq n<k$ and $0\leq j\leq k$.  This corresponds to the Munn tree with $n$ edges to the right of the in-vertex, $k-n$ edges to the left of the in-vertex and the out-vertex is $j$ to the right of the leftmost vertex.  Munn trees of each of the two possible types, together with their normal forms, are given in Figure~\ref{fig_MunnTrees}.

The following lemma is a routine computation.

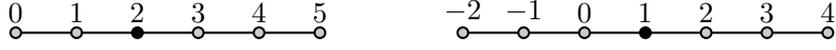
\begin{figure}
% Define style for nodes
%
% Type 1 Munn Tree
%
\begin{tikzpicture}[thick,scale=0.8]
\tikzstyle{lightnode}=[circle, draw, fill=black!20,
                        inner sep=0pt, minimum width=4pt]
\tikzstyle{darknode}=[circle, draw, fill=black!99,
                        inner sep=0pt, minimum width=4pt]
\draw (0,0) node[lightnode] {} --(1,0);
\draw (1,0) node[lightnode] {} --(2,0);
\draw (2,0) node[darknode] {} --(3,0);
\draw (3,0) node[lightnode] {} --(4,0);
\draw (4,0) node[lightnode] {} --(5,0) node[lightnode] {};
\draw (0,0) node[above] {$0$};
\draw (1,0) node[above] {$1$};
\draw (2,0) node[above] {$2$};
\draw (3,0) node[above] {$3$};
\draw (4,0) node[above] {$4$};
\draw (5,0) node[above] {$5$};
%
%
%\draw (0,0) node[below] {$\iota$};
%\draw (2,0) node[below] {$\tau$};
%
%
\end{tikzpicture} \quad \quad \quad
%
% Type 2 Munn Tree
%
\begin{tikzpicture}[thick,scale=0.8]
\tikzstyle{lightnode}=[circle, draw, fill=black!20,
                        inner sep=0pt, minimum width=4pt]
\tikzstyle{darknode}=[circle, draw, fill=black!99,
                        inner sep=0pt, minimum width=4pt]
\draw (0,0) node[lightnode] {} --(1,0);
\draw (1,0) node[lightnode] {} --(2,0);
\draw (2,0) node[lightnode] {} --(3,0);
\draw (3,0) node[darknode] {} --(4,0);
\draw (4,0) node[lightnode] {} --(5,0);
\draw (5,0) node[lightnode] {} --(6,0)
 node[lightnode] {};
\draw (0,0) node[above] {$-2$};
\draw (1,0) node[above] {$-1$};
\draw (2,0) node[above] {$0$};
\draw (3,0) node[above] {$1$};
\draw (4,0) node[above] {$2$};
\draw (5,0) node[above] {$3$};
\draw (6,0) node[above] {$4$};
%
%
%\draw (2,0) node[below] {$\iota$};
%\draw (3,0) node[below] {$\tau$};
%
%
\end{tikzpicture}
\caption{The Munn tree of the left is of Type I and has normal form $x^5y^3$.
The Munn tree on the right is of Type II and has normal form  $x^4y^6x^3$.
In each example, the in-vertex is $0$ and the out-vertex is coloured in black.
} \label{fig_MunnTrees}
\end{figure}

\begin{Lemma}\label{l:norm.forms}
The set of elements of the forms $x^ny^k$ with $0\leq k\leq n$ and $x^ny^kx^j$ with $0\leq n<k$ and $0\leq j\leq k$ constitute a prefix-closed set of normal forms for $M$.
\end{Lemma}

Let $T$ be the spanning tree of $\Gamma$ corresponding to the set of normal forms in Lemma~\ref{l:norm.forms}.  Note that $[1,x^ny^k]$ consists of $n$ $x$-edges followed by $k$ $y$-edges for $0\leq k\leq n$ and $[1,x^ny^kx^j]$ consists of $n$ $x$-edges, followed by $k$ $y$-edges, followed by $j$ $x$-edges for $0\leq n<k$ and $0\leq j\leq k$.  Notice that $T$ is a directed spanning tree rooted at $1$.

 A directed edge of $\Gamma$ is called a \emph{transition edge} if its initial and terminal vertices are in different strongly connected components of $\Gamma$.    Edges of $T$ will be called \emph{tree edges}.

The following lemma is a straightforward computation with Munn trees.

\begin{Lemma}\label{l:basic.props}
The following equalities hold.
\begin{enumerate}
\item $x^ny^kx^{k+1}=x^{n+1}y^{k+1}x^{k+1}$ for $k>n$.
\item $yx^ny^k=x^{n-1}y^nx^{n-k}$ for $n\geq 1$ and $0\leq k\leq n$.
\item $yx^ny^k=x^{n-1}y^k$ if $0<n<k$.
\end{enumerate}
\end{Lemma}

Now we describe which edges of $\Gamma$ are on $T$.

\begin{Prop}\label{p:tree.edges}
The following edges belong to $T$:
\begin{enumerate}
\item $x^n\xrightarrow{\,\,x\,\,} x^{n+1}$ with $n\geq 0$.
\item $x^ny^k\xrightarrow{\,\,y\,\,} x^ny^{k+1}$ with $k\geq 0$.
\item  $x^ny^kx^j\xrightarrow{\,\,x\,\,}x^ny^kx^{j+1}$ with $0\leq n<k$ and $0\leq j<k$.
\end{enumerate}
All remaining edges do not belong to $T$.
\end{Prop}

Next we consider the edges of $\Gamma$ that do not belong to $T$.  We begin with non-transition edges.  We recall  that two elements of $M$ belong to same strongly connected component of the right Cayley graph of the monoid if and only if there is an isomorphism of their underlying Munn trees preserving their in-vertices (but not necessarily their out-vertices). This follows from the description of Green's $\mathscr{R}$ relation in free inverse monoids; see~\cite{Lawson}.

\begin{Prop}\label{p:strong.edges}
An edge of $\Gamma\setminus T$ belongs to a strongly connected component if and only if it is of one of the following two forms:
\begin{enumerate}
\item $x^ny^k\xrightarrow{\,\,x\,\,} x^ny^{k-1}$ with $0<k\leq n$;
\item $x^ny^kx^j\xrightarrow{\,\,y\,\,} x^ny^kx^{j-1}$ with $0\leq n<k$ and $0<j\leq k$.
\end{enumerate}
Moreover, if $e$ is as in (1), then
\[b_e=\ov{(x^ny^{k-1}\xrightarrow{\,\,y\,\,} x^ny^k)(x^ny^k\xrightarrow{\,\,x\,\,} x^ny^{k-1})}\]
and if $e$ is as in (2), then
\[b_e=\ov{(x^ny^kx^{j-1}\xrightarrow{\,\,x\,\,} x^ny^kx^{j})(x^ny^kx^j\xrightarrow{\,\,y\,\,} x^ny^kx^{j-1})}.\]
\end{Prop}
\begin{proof}
Items (1) and (2) are straightforward computations with Munn trees.  The final statements follow by noticing that in the first case $[1,x^ny^{k-1}]$ is an initial segment of $[1,x^ny^k]$ and in the second case  $[1,x^ny^kx^{j-1}]$ is an initial segment of $[1,x^ny^kx^j]$.
\end{proof}

There is only one type of transition edge not belonging to $T$.

\begin{Prop}\label{p:trans.edges}
The transition edges of $\Gamma$ not belonging to $T$ are of the form $x^ny^kx^k\xrightarrow{\,\,x\,\,}x^{n+1}y^{k+1}x^{k+1}$ with $0\leq n<k$.  The corresponding basis element of $H_1(\Gamma)$ is
\[\ov{[x^n,x^{n}y^kx^k]} + (x^ny^kx^k\xrightarrow{\,\,x\,\,}x^{n+1}y^{k+1}x^{k+1})-\ov{[x^n,x^{n+1}y^{k+1}x^{k+1}]}.\]
\end{Prop}
\begin{proof}
The first statement is a straightforward application of Lemma~\ref{l:basic.props} and a Munn tree computation.  The second follows by noting that $[1,x^n]$ is a common initial segment of both $[1,x^ny^kx^k]$ and $[1,x^{n+1}y^{k+1}x^{k+1}]$.
\end{proof}

Our next goal is to assign a weight to the basis element $b_e$ of $H_1(\Gamma)$ corresponding to a directed edge $e$ of $\Gamma\setminus T$.   If $e$ belongs to a strongly connected component of $\Gamma$, then we give $b_e$ weight zero.  If $e$ is as in Proposition~\ref{p:trans.edges}, then we give $b_e$ weight $k$ (which is greater than $0$).  Let $W_k$ be the subgroup of $H_1(\Gamma)$ generated by the $b_e$ of weight at most $k$.  Then we have a strictly increasing chain of subgroups
\[W_0\subsetneq W_1\subsetneq W_2\subsetneq \cdots\]
with $\bigcup_{k\geq 0} W_k=H_1(\Gamma)$.   Our goal is to show that each $W_k$ with $k\geq 0$ is a $\mathbb ZM$-submodule.  Since a finitely generated module cannot be written as the union of a strictly increasing chain of submodules, this will prove that $H_1(\Gamma)$ is not a finitely generated $\mathbb ZM$-module and hence $M$ is not of type $\FP_2$.

We proceed by induction on $k$.

\begin{Prop}\label{p:base.case}
The subgroup $W_0$ is a $\mathbb ZM$-submodule of $H_1(\Gamma)$.
\end{Prop}
\begin{proof}
By Proposition~\ref{p:strong.edges} if $e$ is an edge of weight zero, then $b_e=\ov{p}$ where $p$ is a directed cycle of length $2$.  But any translate of a closed directed path is a closed directed path and hence contained in a strongly connected component of $\Gamma$.  Since every edge of a strongly connected component either belongs to the tree $T$ or has weight zero, we see that $W_0$ is indeed a $\mathbb ZM$-submodule.
\end{proof}

The inductive step is much more technical.
\begin{Prop}\label{p:inductive.step}
For all $k\geq 0$, $W_k$ is a $\mathbb ZM$-submodule of $H_1(\Gamma)$.
\end{Prop}
\begin{proof}
Proposition~\ref{p:base.case} handles the base case of the induction.  Assume that $W_{k-1}$ is a $\mathbb ZM$-submodule and that $k\geq 1$. It suffices to prove that if $e$ is an edge of the form $x^ny^kx^k\xrightarrow{\,\,x\,\,}x^{n+1}y^{k+1}x^{k+1}$ with $0\leq n<k$ and $z\in \{x,y\}$, then $zb_e\in W_k$. By Proposition~\ref{p:trans.edges}, this means we need to show that $ze$ and edges of $z[x^n,x^ny^kx^k]$,  $z[x^n,x^{n+1}y^{k+1}x^{k+1}]$  are of weight at most $k$ or tree edges.

Let us start with $z=y$. In what follows, $x^{-1}$ should be interpreted as $y$; this situation arises when $n=0$.
We consider first $y[x^n,x^ny^kx^k]$.  Note that \[[x^n,x^ny^kx^k]=[x^n,x^ny^n][x^ny^n,x^ny^k][x^ny^k,x^ny^kx^k].\]  By Lemma~\ref{l:basic.props}, we have $yx^n = x^{n-1}y^nx^n$ and $yx^ny^n=x^{n-1}y^n$, which belong to the same strongly connected component.  Thus each edge of $y[x^n,x^ny^n]$ is either a tree edge or an edge of weight zero.  On the other hand, $y[x^ny^n,x^ny^k]$ is a string of $k-n$ $y$-edges from $x^{n-1}y^n$ to $yx^ny^k=x^{n-1}y^k$ (by Lemma~\ref{l:basic.props}) and these are all tree edges. Finally, $y[x^ny^k,x^ny^kx^k]$ is a string of $k$ $x$-edges from $x^{n-1}y^k$ to $x^{n-1}y^kx^k$.  Since $k>n>n-1$, these are again tree edges.

Next, we consider $y[x^n,x^{n+1}y^{k+1}x^{k+1}]$.  Write
\begin{align*}
[x^n,x^{n+1}y^{k+1}x^{k+1}] &= [x^n,x^{n+1}][x^{n+1},x^{n+1}y^{n+1}][x^{n+1}y^{n+1},x^{n+1}y^{k+1}]\\ &\qquad\qquad\cdot[x^{n+1}y^{k+1},x^{n+1}y^{k+1}x^{k+1}]
\end{align*}
 As $yx^n=x^{n-1}y^nx^n$ and  $yx^{n+1}=x^ny^{n+1}x^{n+1}$, by Lemma~\ref{l:basic.props}, we see that $y[x^n,x^{n+1}]=x^{n-1}y^nx^n\xrightarrow{\,\,x\,\,} x^ny^{n+1}x^{n+1}$ is an edge of weight $n<k$ (or a tree edge if $n=0$).  Since $yx^{n+1}=x^ny^{n+1}x^{n+1}$ and $yx^{n+1}y^{n+1} = x^ny^{n+1}$ (see Lemma~\ref{l:basic.props}) belong to the same strongly connected component, we have that $y[x^{n+1},x^{n+1}y^{n+1}]$ consists of tree edges and edges of weight zero.  Next, we have that the translate $y[x^{n+1}y^{n+1},x^{n+1}y^{k+1}]$ is a string of $k-n$ $y$-edges from $yx^{n+1}y^{n+1}=x^ny^{n+1}$ to $yx^{n+1}y^{k+1}=x^ny^{k+1}$, and all these edges are tree edges.  Finally, $y[x^{n+1}y^{k+1},x^{n+1}y^{k+1}x^{k+1}]$ is a string of $k+1$ $x$-edges from $yx^{n+1}y^{k+1}=x^ny^{k+1}$ to $yx^{n+1}y^{k+1}x^{k+1}= x^ny^{k+1}x^{k+1}$. These are again tree edges.

The translate $ye$ is $x^{n-1}y^kx^k\xrightarrow{\,\,x\,\,}x^ny^{k+1}x^{k+1}$, which is an edge of weight $k$, using that $n-1<k$, $yx^ny^kx^k=x^{n-1}y^kx^k$ and $yx^{n+1}y^{k+1}x^{k+1}=x^ny^{k+1}x^{k+1}$ by Lemma~\ref{l:basic.props}, unless $n=0$, in which case it is a tree edge.  This completes the argument that $yb_e\in W_k$.  So we next turn to $z=x$.  There are two cases, $k>n+1$ and $k=n+1$.

Assume first that $k>n+1$.  Then $x[x^n,x^ny^kx^k]=[x^{n+1},x^{n+1}y^kx^k]$ and $x[x^n,x^{n+1}y^{k+1}x^{k+1}]=[x^{n+1},x^{n+2}y^{k+1}x^{k+1}]$ consist of tree edges and $xe=x^{n+1}y^kx^k\xrightarrow{\,\,x\,\,}x^{n+2}y^{k+1}x^{k+1}$ is an edge of weight $k$.  Thus, in this case, $xb_e\in W_k$.

Finally, suppose that $k=n+1$.  Then $xx^ny^kx^k=x^{n+1}y^{n+1}x^{n+1}=x^{n+1}$. Therefore, $x[x^n,x^ny^kx^k]$ is a directed path from $x^{n+1}$ to $x^{n+1}$ and  hence uses only tree edges and edges of weight zero as it is contained in a strongly connected component.  Observe that $xx^{n+1}y^{k+1}x^{k+1}=x^{n+2}y^{n+2}x^{n+2}=x^{n+2}$.  Writing $[x^n,x^{n+1}y^{k+1}x^{k+1}]=[x^n,x^{n+1}][x^{n+1},x^{n+1}y^{k+1}x^{k+1}]$, we see that $x[x^n,x^{n+1}y^{k+1}x^{k+1}]$ is the concatenation of the tree edge $x^{n+1}\xrightarrow{\,\,x\,\,}x^{n+2}$ with a directed path from $x^{n+2}$ to itself and the latter path uses only tree edges and edges of weight zero as it is contained in a strongly connected component.  Also, we have that $xe=x^{n+1}\xrightarrow{\,\,x\,\,}x^{n+2}$ is a tree edge.  We conclude that $xb_e\in W_k$ in this case as well.  This completes the proof that $W_k$ is a $\mathbb ZM$-submodule of $H_1(\Gamma)$.
\end{proof}

Proposition~\ref{p:inductive.step} completes the proof of Theorem~\ref{t:monogenic.case} in light of the discussion preceding Proposition~\ref{p:base.case}.

\end{document}